\documentclass[12pt]{article}

\usepackage{amsmath}
\usepackage{amsfonts,amssymb,amsthm,overpic}
\makeatletter
\newcommand{\diam}{\mathop{\operator@font diam}}
\usepackage{setspace}
\usepackage{multicol}
\usepackage{multirow}
\usepackage[compact]{titlesec}
\usepackage{xcolor}

\usepackage{epsfig}

\newtheorem{definition}{Definition}[section]

\newtheorem{theorem}{Theorem}[section]

\begin{document}

\title{\Huge{\textsc{On Sliced Spaces: Global Hyperbolicity revisited}}}

\author{Kyriakos Papadopoulos$^1$, Nazli Kurt$^2$, Basil K. Papadopoulos$^3$\\
\small{1. Department of Mathematics, Kuwait University, PO Box 5969, Safat 13060, Kuwait}\\
\small{2. Open University, UK}\\
\small{3. Department of Civil Engineering, Democritus University of Thrace, Greece}\\
E-mail: \textrm{ kyriakos@sci.kuniv.edu.kw}
}

\date{}

\maketitle

\begin{abstract}
We give a topological condition for a generic sliced space to be globally hyperbolic, without any hypothesis on the lapse function, shift function and spatial metric.
\end{abstract}

\section{Preliminaries.}

We begin with the definition of a {\em sliced space}, that one can read in \cite{Cotsakis}, as a continuation of a study in \cite{Choquet} and \cite{ChoquetCotsakis} on systems of Einstein equations.

Let $V = M \times I$, where $M$ is an $n$-dimensional smooth manifold and $I$ is an interval of the real line, $\mathbb{R}$. We equip $V$ with a $n+1$-dimensional Lorentz metric $g$, which splits in the following way: 
\[g = -N^2 (\theta^0)^2 + g_{ij} \theta^i \theta^j,\]
where $\theta^0 = dt$, $\theta^i = dx^i + \beta^i dt$,
$N = N(t,x^i)$ is the {\em lapse function}, $\beta^i(t,x^j)$
is the {\em shift function} and $M_t = M \times \{t\}$, spatial
slices of $V$, are spacelike submanifolds equipped with the time-dependent spatial metric $g_t = g_{ij} dx^i dx^j$. Such a product space $V$ is called a {\em sliced space}. 


Throughout the paper, we will consider $I = \mathbb{R}$.

The author in \cite{Cotsakis} considered sliced spaces with uniformly bounded lapse, shift and spatial metric; by this hypothesis, it is ensured that parametre $t$ measures up to a positive factor bounded (below and above) the time along the normals to spacelike slices $M_t$, the $g_t$ norm of the shift vector $\beta$ is uniformly bounded by a number and the time-dependent metric $g_{ij} dx^i dx^j$ is uniformly bounded (below and above) for all $t \in I (=\mathbb{R})$, respectively.  

Given the above hypothesis, in the same article the following theorem is proved.

\begin{theorem}[Cotsakis] Let $(V,g)$ be a sliced space with uniformly bounded lapse $N$, shift $\beta$ and spatial metric $g_t$. Then, the following are equivalent:

\begin{enumerate}
\item $(M_0,\gamma)$ a complete Riemannian manifold.
\item The spacetime $(V,g)$ is globally hyperbolic.
\end{enumerate}
\end{theorem}

In this article we review global hyperbolicity of sliced spaces, in terms of the product topology defined on the space $M \times \mathbb{R}$, for some finite dimensional smooth manifold $M$.

\section{Strong Causality of Sliced Spaces.}

Let $(V = M \times \mathbb{R},g)$ be a sliced space. Consider  the product topology $T_P$, on $V$. A base for $T_P$ consists of all sets of the form $A\times B$, where $A \in T_M$ and $B \in T_{\mathbb{R}}$. Here $T_M$ denotes the natural topology of the manifold $M$ where, for an appropriate Riemann metric $h$, it has a base consisting of open balls $B_{\epsilon}^h (x)$ and $T_{\mathbb{R}}$ is the usual topology on the real line,  with a base consisting of open intervals $(a,b)$.
For trivial topological reasons, we can restrict our discussion on $T_P$ to basic-open sets $B_{\epsilon}^h (x) \times (a,b)$, which can be intuitively called as ``open cylinders'' in $V$.

We remind the Alexandrov topology $T_A$ (see \cite{Penrose-difftopology}) has a base consisting of open sets of the form $<x,y> = I^+(x) \cap I^-(y)$, where $I^+(x) = \{z \in V : x \ll z\}$ and $I^-(y) = \{z \in V : z \ll y\}$, where $\ll$ is the {\em chronological order} defined as $x \ll y$ iff there exists a future oriented timelike curve, joining $x$ with $y$. By $J^+(x)$ one denotes the topological closure of $I^+(x)$ and by $J^-(y)$ that one of $I^-(y)$.

We use the definition of strong causality and global hyperbolicity from \cite{Penrose-difftopology}; global hyperbolicity is an important causal condition in a spacetime related to major problems such as spacetime singularities, cosmic cencorship etc.

\begin{definition}
A spacetime is globally hyperbolic, iff it is strongly causal and the ``causal diamonds'' $J^+(x) \cap J^-(y)$ are compact.
\end{definition}

We prove the following theorem.

\begin{theorem}\label{1}
Let $(V,g)$ be sliced space. Then, the following are equivalent.
\begin{enumerate}
\item $V$ is strongly causal.

\item $T_A \equiv T_P$.

\item $T_A$ is Hausdorff.
\end{enumerate}
\end{theorem}
\begin{proof}
2. implies 3. is obvious and that 3. implies 1. can be found in \cite{Penrose-difftopology}.

For 1. implies 2.: That $T_A$ is coarser than $T_P$ is trivial. Now, we consider $\epsilon > 0 $, such that $B_{\epsilon}^h (A) \in T_M$, so that $B_{\epsilon}^h (A) \times (a,b) = B \in T_P$. We let strong causality hold at an event $P$  and consider $P \in B \in T_P$. We show that there exists $<X,Y>\in T_A$, such that $P \in <X,Y> \subset B$. Now, consider a simple region $R$ in $<X,Y>$ which contains $P$ and $P \in Q$, where $Q$ is a causally convex-open subset of $R$. Thus, we have $U,V \in Q$, such that $P \in <U,V> \subset Q$. Finally, $P \in <U,V> \subset Q \subset B$ and this completes the proof.
\end{proof}

\section{Global Hyperbolicity of Sliced Spaces, Revisited.}

\begin{theorem}\label{2}
Let $(V,g)$ be a sliced space, where $V = M\times \mathbb{R}$, $M$ is an $n$-dimensional manifold ($n\ge 2$) and $g$ the $n+1$ Lorentz metric in $V$. Let $T_A$ be the Alexandrov spacetime topology on $V$ and $T_P$ the product topology on $V$. Then, $(V,g)$ is globally hyperbolic, iff for every basic-open set in $D$ in $T_A$ there exists a basic-open set $B$ in $T_P$, such that $D \subset B$. 
\end{theorem}
\begin{proof}
If $V$ is globally hyperbolic, then it is strongly causal (see \cite{Minguzzi}) and from Theorem \ref{1} we get that $T_P \equiv T_A$. Consider a basic-open set (an ``open diamond'') $D$, in $T_A$; then, $D \in T_P$. Choose an open diamond $B \in T_P$, such that $D \subset B$; the result follows.

Now, consider $D =I^+(x) \cap I^-(y) $ be a basic-open set in $T_A$. Then, there exists $B^h(x)$ in the manifold topology on $M$ and $(a,b)$ in the usual topology of $\mathbb{R}$, such that $D \subset B^h(x) \times (a,b) = B$. But, The closure of $B$ is compact (as a product of compact sets) and so the closure of $D$ will be compact too, as a closed subset of a compact set.

\end{proof}

We note that neither in Theorem \ref{1} nor in Theorem \ref{2} we made any hypothesis on the lapse function, shift function or on the spatial metric.

\end{document}